\documentclass[a4paper]{article}

\usepackage{amsmath,amssymb,amsfonts,amsthm,hyperref,subfigure}
\usepackage[textsize=footnotesize,backgroundcolor=yellow!10,linecolor=gray!35]{todonotes}

\theoremstyle{plain}
\newtheorem{prop}{Proposition}

\theoremstyle{remark}

\newtheorem{exmp}{Example}

\DeclareMathOperator*{\argmin}{arg\,min}
\DeclareMathOperator*{\essinf}{ess\,inf}

\makeatletter
\newcommand{\logmessage}[1]{\@latex@warning{#1}}
\makeatother
\IfFileExists{../Bibliographie/Users_Guide.txt}{
  \def\BibPath{../Bibliographie/}}{
    \IfFileExists{../../Bibliographie/Users_Guide.txt}{
      \def\BibPath{../../Bibliographie/}}{
        \IfFileExists{../../../Bibliographie/Users_Guide.txt}{
          \def\BibPath{../../../Bibliographie/}}{
            \IfFileExists{../../../../Bibliographie/Users_Guide.txt}{
              \def\BibPath{../../../../Bibliographie/}}{
                \IfFileExists{../../../../../Bibliographie/Users_Guide.txt}{
                  \def\BibPath{../../../../../Bibliographie/}}{
                    \IfFileExists{../../../../../../Bibliographie/Users_Guide.txt}{
              		  \def\BibPath{../../../../../../Bibliographie/}}{
    \logmessage{Directory 'Bibliographie' not found}
      \def\BibPath{}}}}}}}

\begin{document}
\title{Convective regularization for optical flow}
\author{Jos\'{e} A.~Iglesias\footnote{Computational Science Center, University of Vienna, Vienna, Austria}, Clemens Kirisits\footnote{Johann Radon Institute for Computational and Applied Mathematics (RICAM), Austrian Academy of Sciences, Linz, Austria}}
\maketitle

\begin{abstract}
We argue that the time derivative in a fixed coordinate frame may not be the most appropriate measure of time regularity of an optical flow field.
Instead, for a given velocity field $v$ we consider the convective acceleration $v_t + \nabla v v$ which describes the acceleration of objects moving according to $v$.
Consequently we investigate the suitability of the nonconvex functional $\|v_t + \nabla v v\|^2_{L^2}$ as a regularization term for optical flow.
We demonstrate that this term acts as both a spatial and a temporal regularizer and has an intrinsic edge-preserving property.
We incorporate it into a contrast invariant and time-regularized variant of the Horn-Schunck functional, prove existence of minimizers and verify experimentally that it addresses some of the problems of basic quadratic models.
For the minimization we use an iterative scheme that approximates the original nonlinear problem with a sequence of linear ones.
We believe that the convective acceleration may be gainfully introduced in a variety of optical flow models.
\end{abstract}

\section{Introduction}

\paragraph{Motivation.}
Optical flow is the apparent motion in a sequence of images and can be described by a velocity field. Using variational techniques to estimate this velocity field requires the design of an appropriate energy functional. Typically such a functional is a sum of two parts. The first part, sometimes called data term, measures the accuracy with which the velocity field describes the observable image motion. By ensuring consistency of the flow field, the second part, also called regularization term, gives information the data term cannot provide and thereby addresses the inherent ill-posedness of the optical flow problem. Naturally, related research has to a large degree been concerned with tuning the second term to optimally capture the characteristics of velocities of real-world image sequences.

Generally speaking regularization terms for optical flow fall into two categories: those which only penalize spatial derivatives of the velocity field, and those where derivatives in both space and time are penalized. The first category is by far the more popular one. Its developments have closely followed those of variational image restoration, because the regularity of images dictates to a significant degree the spatial regularity of velocity fields describing their motion. In other words, image discontinuities tend to coincide with motion discontinuities. These, it was found, can be appropriately dealt with using subquadratic or anisotropic regularization.

Research on the second category, i.e.\ that of spatiotemporal regularization, is relatively scarce. Again this is in analogy with the field of image processing, where the joint restoration of image sequences is quite under-represented as compared to the restoration of single frames. The inclusion of time derivatives has the obvious disadvantage of turning a series of decoupled 2D problems into one 3D problem. In the early days of computer vision limited computer memory and processing power was prohibitive to such approaches. Clearly this is no longer the case. On the contrary, recent publications have articulated and addressed the need for time-coupled models even for data with three instead of two space dimensions. See for example \cite{AmaMyeKel13,SchmShaScheWebThi13}.


\begin{figure}[h]
\centering{
\resizebox{55mm}{!}{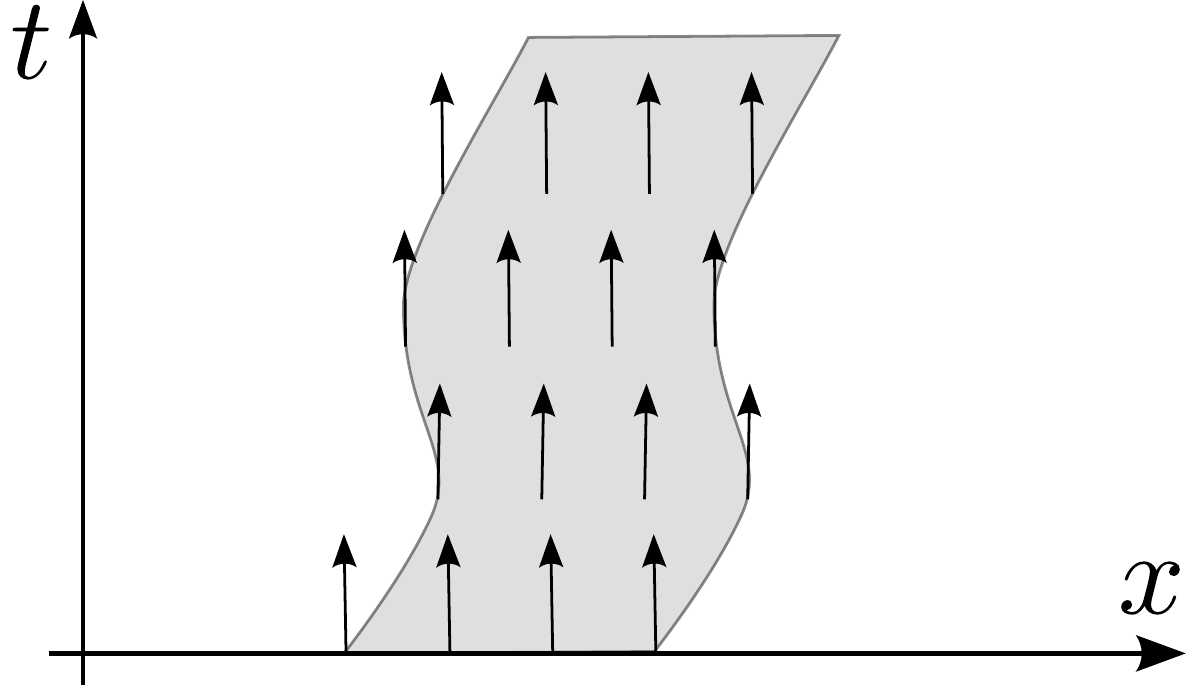}
\caption{Using partial time derivatives in flow regularization can blur across object boundaries. Here, the shaded area represents the space-time signature of a moving object and the time direction points both inside and outside of it at different points.}
\label{fig:timevectors}
}
\end{figure}

It is well-known that indiscriminate smoothing in space can lead to blurring at the boundaries of objects. This statement is equally true in the space-time domain, if we identify objects with their space-time signature (Figure \ref{fig:timevectors}). In particular, blind penalization of the partial time derivative of a velocity field can lead to a loss of accuracy, especially at the boundaries of \emph{moving} objects. Therefore we suggest to take the time derivative only along the movement of the object.

This derivative has a very natural physical interpretation. Let $v(t,x)$ be a time dependent velocity field on $\mathbb{R}^d$, and let $t \mapsto \gamma(t;x_0)$ be the trajectory of a certain object initially located at $x_0 \in \mathbb{R}^d$ that moves according to $v$. The translation into mathematical terms of this connection between $\gamma$ and $v$ is the following initial value problem
\begin{equation}\label{eq:ivp}
\begin{aligned}
	\gamma'(t;x_0) &= v (t,\gamma(t;x_0)) \\
	\gamma(0;x_0) &= x_0.
\end{aligned}
\end{equation}
Taking time derivatives on both sides of the differential equation above gives a formula in terms of $v$ for the acceleration of the object moving along $\gamma$
\begin{align*}
	\gamma''(t;x_0)  &= \frac{d}{dt} v (t,\gamma(t;x_0)) \\
					&= \frac{\partial}{\partial t} v(t,\gamma(t;x_0)) + \nabla v (t,\gamma(t;x_0)) \, v(t,\gamma(t;x_0))
\end{align*}
This expression, which is sometimes called \emph{convective acceleration}, is to be contrasted with the partial time derivative of $v$ evaluated at $(t,\gamma(t;x_0))$
\begin{equation*}
	\frac{\partial}{\partial t} v (t,\gamma(t;x_0)),
\end{equation*}
which has a notably different physical meaning. In this article we argue that in some situations the former can be more appropriate for the regularization of optical flow.

This work is devoted to investigating the suitability of the functional
\begin{equation}\label{eq:convterm}
	v \mapsto \|v_t + \nabla v v \|^2_{L^2}
\end{equation}
as a smoothness term for optical flow. We believe a sensible way to do so is to include this term into a simple optical flow model and to compare the results with those of the same model without the new term. The reference model we choose for this task is a contrast invariant and time-regularized version of the Horn-Schunck functional \cite{HorSchu81}. Incorporating \eqref{eq:convterm} into our model comes, however, at the cost of leading to a nonconvex functional and to nonlinear optimality conditions. Numerically, we approximate the problem by a sequence of convex quadratic ones.

We stress that it is not the aim of the proposed model to compete in the highest ranks of motion estimation benchmarks. Instead we want to point out a certain aspect of optical flow regularization, which we feel has not received its due amount of attention. Even though it leads to a more challenging variational problem, we demonstrate that a reasonable approximation is numerically tractable and pays off in terms of accuracy.

\paragraph{Related work.}
Horn and Schunck \cite{HorSchu81} are generally credited with having laid the groundwork for variational optical flow models. Penalizing the squared $L^2$-norms of the spatial derivatives, their model regularizes the optical flow isotropically as well as homogeneously over the image domain. In \cite{Nag83} Nagel proposed to suppress smoothing across image discontinuities by only penalizing the derivatives of the vector field along the level lines of the image. This idea was subsequently extended to the space-time domain in \cite{Nag90}.

A different route to spatiotemporal regularization was taken by Weickert and Schn\"{o}rr. In \cite{WeiSchn01b} they used an isotropic but essentially subquadratic regularizer. While still convex this approach leads to nonlinear optimality conditions. In the survey \cite{WeiSchn01a} the same authors classify convex spatial as well as spatiotemporal regularizers for optical flow. According to their nomenclature our approach would classify as anisotropic and flow-driven. It is, however, fundamentally different from the regularizers considered there, not only because of nonconvexity. In \cite{WeiSchn01a} the anisotropy of flow-driven regularizers is determined by the Jacobian of the velocity field, whereas in our approach it is determined by the velocity itself.

Chaudhury and Mehrotra pursue an interesting trajectory-based regularization strategy \cite{ChaMeh95}. Motivated by the principle of least action and the inertia of motion they postulate that both length and curvature of trajectories of moving objects should be minimal. There is a close connection between a trajectory's curvature and the convective acceleration of an object moving along that trajectory. We discuss this connection in the next section. More recently, Salgado and S\'{a}nchez \cite{SalSan07} as well as Volz and coauthors \cite{VolBruValZim11} have proposed time-discrete models using explicit trajectorial regularization. The present article tries to capture the essence of trajectorial regularization for optical flow in an entirely continuous setting.

Finally, we note that for the problem of image sequence reconstruction it is not uncommon to regularize the sequence along the optical flow, which can be either precomputed or estimated simultaneously. This derivative along the optical flow is just the convective derivative of the image sequence. See
\cite{CocChaBla03,PreDroGarTelRum08} for example.

\paragraph{Outline.} In Section \ref{sec:model} we introduce our model. In Sec.~\ref{sec:conacc} we discuss the convective acceleration of vector fields. Section \ref{sec:conreg} contains properties of the functional $\|v_t + \nabla v v \|^2_{L^2}$ as a regularization term. The last part of Sec.~\ref{sec:model} treats our choice of data term. Section \ref{sec:numerics} is dedicated to discussing our numerical minimization approach. Finally, we present experimental results in Section \ref{sec:experiments}.

\section{Model}\label{sec:model}
\subsection{Convective acceleration}\label{sec:conacc}
\paragraph{Notation.} Let $\Omega\subset\mathbb{R}^2$ be a bounded Lipschitz domain, $T>0$, and set $ E = (0,T) \times \Omega$. We denote points in $\Omega$ by $x=(x^1,x^2)^\top$. Let $\phi: E \to \Omega$ be a flow on $\Omega$: That is, for fixed $t \in (0,T)$ the map $\phi(t,\cdot)$ is a diffeomorphism of $\Omega$, while for fixed $x_0\in \Omega$ the trajectory described by $\phi(\cdot,x_0)$ is smooth. The vector field on $E$ that gathers the velocities of all trajectories associated to $\phi$ is defined by
\begin{equation}\label{eq:velocity}
	u(t,x) = \phi_t(t,x_0) \Big|_{x_0 = \phi^{-1}(t,x)}
\end{equation}
where $\phi_t$ is the partial derivative of $\phi$ with respect to its first argument and $\phi^{-1}(t,x)$ denotes the inverse of $\phi(t,\cdot)$ evaluated at $x$. Notice that $u(t,x)$ describes the velocity of the trajectory passing through $x$ at the time $t$, so $u(s,x)$ in general corresponds to a different trajectory, if $s \neq t$.


Let $f : E \to \mathbb{R}^N$ be a possibly vector-valued quantity on $E$. We define the convective derivative of $f$ along $\phi$ as
\begin{equation} \label{eq:conder}
\begin{aligned}
	D_u f(t,x)	&= \frac{d}{dt} f(t,\phi(t,x_0)) \Big|_{x_0 = \phi^{-1}(t,x)} \\
				&=	f_t(t,x) + \nabla f(t,x) u(t,x).
\end{aligned}
\end{equation}
Here $\nabla f = (f_{x^1}, f_{x^2}) \in \mathbb{R}^{N \times 2}$ is the spatial gradient of $f$. Notice that the convective derivative only depends on $u(t,x)$, thereby justifying the notation $D_u f$. Using the notation $\bar \nabla f = (f_t, \nabla f)$ and $\bar u = (1, u^\top)^\top$ we can write $D_u f = \bar \nabla f \bar u$. If we set $f = u$, the resulting vector field
\begin{equation}\label{eq:conacc}
	D_u u = u_t + \nabla u u = \bar \nabla u \bar u
\end{equation}
is called the \emph{convective acceleration} of the flow $\phi$.

\paragraph{Flows with vanishing convective acceleration.}
When proposing a new regularization term, it is always helpful to know when it is minimal. Therefore, we briefly discuss what conditions must be met so that a vector field $u$ satisfies
\begin{equation}\label{eq:burgers}
	D_u u = 0, \quad \text{for all } (t,x) \in E,
\end{equation}
and give a few examples afterwards. Intuitively $D_u u = 0$ means that $u$ satisfies a ``flow constancy condition". This is to be understood as an analogy to the widely known \emph{brightness constancy condition}, which is frequently used to derive the optical flow equation. Note that with the notation introduced in the previous paragraph, for an image sequence $f$, the optical flow equation reads $D_u f = 0$.

From the definition of the convective derivative \eqref{eq:conder} and \eqref{eq:velocity} it follows that
\begin{equation*}
	D_u u (t,\phi(t,x_0)) = \phi_{tt} (t,x_0),
\end{equation*}
for all $x_0 \in \Omega$ and $t\in(0,T)$. Now fix an $x_0 \in \Omega$ and, for simplicity, denote the trajectory $\phi(\cdot,x_0)$ by $\gamma(\cdot)$, so that we have $D_u u = \gamma''$. Assuming that $\gamma'$ does not vanish on $(0,T)$, we can write
\begin{align*}
	\gamma'' 	&= \frac{d}{dt} \left( \left|\gamma'\right| \frac{\gamma'}{\left|\gamma'\right|} \right) \\
				&= \frac{d \left|\gamma' \right| }{dt} \frac{\gamma'}{\left|\gamma'\right|} + \left|\gamma'\right| \frac{d}{dt} \frac{\gamma'}{\left|\gamma'\right|}.
\end{align*}
Since the unit tangent vector $\gamma'/|\gamma'|$ is always perpendicular to its derivative, the convective acceleration $\gamma''$ is a sum of two mutually orthogonal vectors. Such a sum can only vanish, if both vectors vanish. Due to the assumption that $\gamma' \neq 0$, we conclude that
\begin{align*}
	\frac{d \left|\gamma' \right| }{dt} &= 0, \quad \text{and} \\
	\frac{d}{dt} \frac{\gamma'}{\left|\gamma'\right|} &= 0.
\end{align*}
This proves the quite intuitive fact that, if a flow has vanishing convective acceleration, then all its trajectories must be straight lines and have constant speed. Below we give a few examples of such vector fields.

\begin{exmp}\label{ex:zeroconacc}
	Let $g:\mathbb{R}\to\mathbb{R}$ be a differentiable function. The two vector fields $v_1,v_2: E \rightarrow \mathbb{R}^2$ defined by
	\begin{equation*}
		v_1(t,x^1,x^2) = \begin{pmatrix}0 \\ g(x^1) \end{pmatrix}, \quad
		v_2(t,x^1,x^2) = \begin{pmatrix} g(x^2) \\ 0 \end{pmatrix}
	\end{equation*}
	do not depend on time and satisfy $\nabla v_i v_i = 0$. Thus, by formula \eqref{eq:conacc} they have vanishing convective acceleration. Their integral curves are parallel to the $x_2$- and $x_1$-axes, respectively.
\end{exmp}

\begin{exmp}
	Let $0 \notin \Omega$. Denote by $e_r$ the unit vector in radial direction and let $g:\mathbb{R}\to\mathbb{R}$ be a differentiable function with period $2\pi$. The vector field $v:E\to\mathbb{R}^2$ which in polar coordinates is given by
	\begin{equation*}
		v(t,r,\varphi) = g(\varphi) e_r
	\end{equation*}
	has zero convective acceleration. In contrast to the previous example the integral lines of $v$ are not all mutually parallel.
\end{exmp}

\begin{exmp}\label{ex:burgers}
	The vector fields in the previous examples were all constant in time. Time-dependent examples can be constructed from time-dependent solutions $u$ of the inviscid Burgers' equation in one space dimension (\cite{Eva10}, Section 3.4). That is, let $u:E\to\mathbb{R}$ solve \eqref{eq:burgers} with $E = (0,T) \times \mathbb{R}$. Then
	\begin{equation*}
		v_1(t,x^1,x^2) = \begin{pmatrix} u(t,x^1) \\ 0 \end{pmatrix}, \quad
		v_2(t,x^1,x^2) = \begin{pmatrix} 0 \\ u(t,x^2) \end{pmatrix}
	\end{equation*}
	satisfy $D_{v_i} v_i = 0$ while having partial time derivatives that do not vanish identically. As in the first example the integral curves are mutually parallel straight lines. The main difference, however, is that two trajectories passing through the same point in space but at different times, might do so at different speeds.
\end{exmp}

\paragraph{Relation to curvature.}
Given a curve $\gamma : (0,T) \to \Omega$, we can define its curvature as the normal part of the arclength derivative of the unit tangent vector
\begin{align*}
	\kappa_\gamma	= \left( \frac{\gamma'}{|\gamma'|}\right)^\perp \cdot \frac{1}{|\gamma'|} \frac{d}{dt} \frac{\gamma'}{|\gamma'|}
					= \frac{\gamma'^\perp \cdot \gamma''}{|\gamma'|^3},
\end{align*}
where $(a,b)^{\top\perp} = (-b,a)^\top$. If $\gamma$ is arc length parametrized, that is $|\gamma'|\equiv 1$, then the expression simplifies to $\kappa_\gamma = |\gamma ''|$,  because in that case $\gamma'^\perp = \gamma''/|\gamma''|$.

For a given $x_0 \in \Omega$ we can apply the formula above to the trajectory $t \mapsto \phi(t, x_0)$ of a flow $\phi$. Recall that in this case $\gamma'$ becomes $u$ and $\gamma''$ becomes $D_u u$. Thus we get
\begin{equation*}
	\kappa_{\phi(\cdot,x_0)} = \frac{u^\perp \cdot D_uu}{|u|^3}.
\end{equation*}
In particular, whenever $|u(t,\phi(t,x_0))| \equiv 1$ on some subinterval of $(0,T)$, then the norm of the convective acceleration $|D_u u|$ coincides with the absolute value of the curvature of $\phi(\cdot, x_0)$ on this subinterval
\begin{equation*}
	\kappa_{\phi(\cdot,x_0)} = \pm|D_u u|.
\end{equation*}


\subsection{Convective regularization}\label{sec:conreg}
In this section we study some properties of the functional
\begin{equation}\label{eq:conaccreg}
	u \mapsto \frac{1}{2}\|D_u u\|^2_{L^2}
\end{equation}
as a regularization term for optical flow. By $\| \cdot \|_{L^2}$ we mean the norm of $L^2(E,\mathbb{R}^2)$. Below we continue to use this shorthand whenever convenient.

\paragraph{Interpretation of the convective term.}
Functional \eqref{eq:conaccreg} has two interpretations, one as a smoothing term and and another one as a projection term. Recalling the bar-notation we introduced in Section \ref{sec:conacc}, the integrand of \eqref{eq:conaccreg} can be written as a quadratic form in the partial derivatives of $u$
\begin{equation}\label{eq:sqrnormconvacc}
	|D_u u|^2 = \bar\nabla u^1 \bar u \bar u^\top (\bar\nabla u^1)^\top + \bar\nabla u^2 \bar u \bar u^\top (\bar\nabla u^2)^\top.
\end{equation}
Observe that the diffusion tensor
\begin{equation*}
	\bar u \bar u^\top = \begin{pmatrix}1 & u^\top \\ u & uu^\top \end{pmatrix}
\end{equation*}
is a projection matrix onto the span of $\bar u$ composed with a scaling by $|\bar u|^2$. Therefore, minimization of \eqref{eq:conaccreg} leads to smoothing of the vector field $u$ only in direction $\bar u$, where precedence is given to regions where the magnitude of $u$ is relatively large. Clearly, since $\bar u \neq (1,0,0)^\top$ in general, the proposed convective regularization term is not a purely temporal regularizer, but it also enforces spatial smoothness of the vector field in a way that is consistent with the motion. Figure \ref{fig:flowvectors} illustrates this behavior.

\begin{figure}[h]
\centering{
\resizebox{55mm}{!}{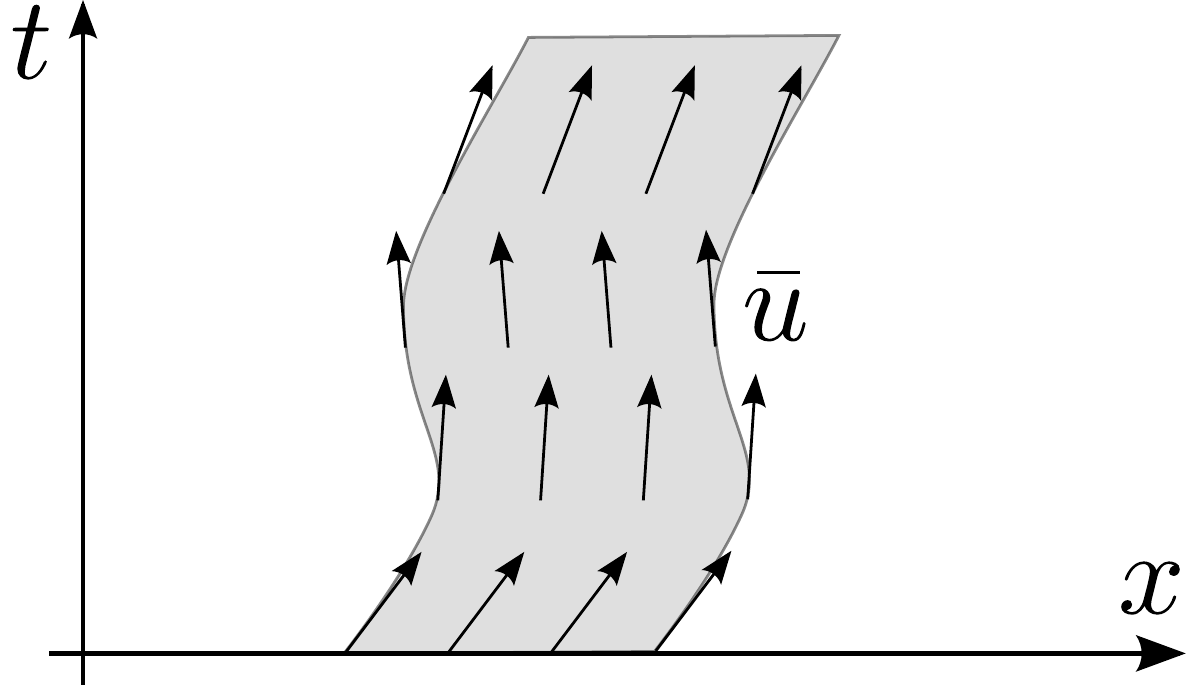}
\caption{The shaded area represents the space-time signature of a moving object. The arrows not only indicate the space-time velocity $\bar u$ of the object but also the direction of diffusion enforced by the tensor $\bar u \bar u^\top$. As opposed to the situation with the time derivative (Figure \ref{fig:timevectors}), object boundaries are respected.}
\label{fig:flowvectors}
}
\end{figure}

On the other hand, the roles of $\bar u$ and $\bar\nabla u^i$ can be exchanged so that \eqref{eq:sqrnormconvacc} rewrites as a sum of two quadratic forms in $\bar u$
\begin{equation}\label{eq:sqrnormconvacc2}
	|D_u u|^2 = \bar u^\top (\bar\nabla u^1)^\top \bar\nabla u^1 \bar u + \bar u^\top (\bar\nabla u^2)^\top \bar\nabla u^2 \bar u.
\end{equation}
The matrix $(\bar\nabla u^i)^\top \bar\nabla u^i$ projects onto the span of $\bar\nabla u^i$ and scales by the squared magnitude of $\bar\nabla u^i$. From this point of view minimizing \eqref{eq:conaccreg} amounts to forcing $\bar u$ to be orthogonal to the directions of greatest change of both $u^1$ and $u^2$, where precedence is again given whenever this change is relatively drastic. Note that the outer product of a vector with itself is always a positive semidefinite matrix, so that there can be no cancellations in \eqref{eq:sqrnormconvacc}, \eqref{eq:sqrnormconvacc2}.

The concurrence of the two different interpretations above is made explicit by looking at the Euler-Lagrange equation of \eqref{eq:conaccreg}, given by
\begin{equation}\label{eq:convectiveeullag}
	\left( (\nabla u^1)^\top \bar\nabla u^1 + (\nabla u^2)^\top \bar\nabla u^2 \right) \bar u
	- \bar \nabla \cdot \left( \bar \nabla u \bar u \bar u^\top \right) = 0.
\end{equation}
Here, the space-time divergence operator $\bar \nabla \cdot {}$ acts along rows.

\paragraph{Variational properties.}
In the following we adopt the notation $A \lesssim B$ whenever $A \le B$ holds up to multiplication by a positive constant. According to the basic estimate
\begin{equation*}
	\|D_u u\|_{L^2}^2 \lesssim \| u_t \|_{L^2}^2 + \|u\|^2_{L^\infty} \| \nabla u \|^2_{L^2}
\end{equation*}
a seemingly natural space over which to minimize, for a given image sequence $f$, the optical flow functional
\begin{equation*}
	\mathcal{F}(u) = \| D_u f \|_{L^2}^2 + \alpha\|D_u u\|_{L^2}^2
\end{equation*}
would be $X = H^1(E,\mathbb{R}^2) \cap L^\infty(E,\mathbb{R}^2)$. Now, in order to show existence of minimizers using the direct method \cite{Dac08}, we need to ensure existence of a minimizing sequence which converges. This is typically done by establishing a coercivity condition of the type
\begin{equation*}
	\mathcal{F}(u) \gtrsim \, \| u \|^p_X - b
\end{equation*}
for some $b\ge 0$ and $p\ge 1$. In the examples above we have seen, however, that there are vector fields $u$ with vanishing convective acceleration but nonzero partial derivatives. This implies that not even the inequality
\begin{equation*}
	\|D_u u\|_{L^2}^2 \gtrsim \, \| \bar \nabla u \|^p_{L^2} - b
\end{equation*}
holds true. Therefore, in order to guarantee existence of minimizers we consider the functional
\begin{equation}\label{eq:functional}
	\begin{gathered}
		\mathcal{E} : H^1(E,\mathbb{R}^2) \to [0,+\infty] \\
		\mathcal{E}(u) = \|\lambda D_u f\|^2_{L^2} + \alpha \|D_u u\|^2_{L^2} + \beta \|\bar \nabla u\|^2_{L^2},
	\end{gathered}
\end{equation}
where $\alpha, \beta >0$ and $\lambda : E \to \mathbb{R}_+$ is a weighting function which is specified in Sec.~\ref{sec:dataterm}. In addition we have to make the assumption introduced in \cite{Schn91a} that the partial derivatives of $f$ are linearly independent in $L^2(E)$, that is
\begin{equation}\label{eq:linind}
	\langle f_{x^1},f_{x^2} \rangle_{L^2} < \| f_{x^1} \|_{L^2}  \| f_{x^2} \|_{L^2}.
\end{equation}

\begin{prop}\label{thm:existence}
Let $f\in W^{1,\infty}(E)$ satisfy \eqref{eq:linind} and assume that $\lambda \in L^\infty(E)$ is such that $\essinf \lambda > 0$. Then the minimization of $\mathcal{E}$ over $H^1(E,\mathbb{R}^2)$ has at least one solution.
\end{prop}
\begin{proof} We show that $\mathcal{E}$ is proper, coercive and lower semicontinuous with respect to the weak topology of $H^1(E,\mathbb{R}^2)$. That $\mathcal{E}$ is proper follows from nonnegativity and the assumptions on $f$ and $\lambda$.

Next we prove the coercivity estimate
\begin{equation*}
	\mathcal{E} (u) \gtrsim \| u \|^2_{H^1} - b
\end{equation*}
for $b\ge 0$. We obviously have $\mathcal{E} (u) \ge \beta \| \bar \nabla u \|^2_{L^2}$. So it remains to show that $\mathcal{E} (u) \gtrsim \| u \|^2_{L^2} - b$. Denoting the average of $u$ by $u_E = 1/|E| \int_E u$ we have
\begin{equation}\label{eq:chain}
\begin{aligned}
	\| u \|^2_{L^2}
		&\lesssim \| u_E\|^2_{L^2} + \|u - u_E \|^2_{L^2} \\
		&\lesssim \| u_E \cdot \nabla f\|_{L^2}^2 + \| \bar \nabla u \|_{L^2}^2 \\
		&\lesssim \| u \cdot \nabla f\|_{L^2}^2 + \| (u-u_E) \cdot \nabla f\|_{L^2}^2 + \| \bar \nabla u \|_{L^2}^2 \\
		&\lesssim \| D_u f\|_{L^2}^2 + \|f_t\|^2_{L^2} + \| \bar \nabla u \|_{L^2}^2 \\
		&\lesssim \mathcal{E}(u) + \|f_t\|^2_{L^2},
\end{aligned}
\end{equation}
which proves coercivity. There are three main ingredients in the above estimate. The first one is a quadratic variant of the triangle inequality
$$ \| v + w \|^2 \le 2 \| v \|^2 + 2 \|w \|^2. $$
The second one
$$ \| u_E\|^2_{L^2} \lesssim \| u_E \cdot \nabla f\|^2_{L^2} $$
uses the assumption of linear independence of $f_{x^1}$ and $f_{x^2}$ and is proved in \cite[p.~29]{Schn91a}. The third ingredient is the Poincar\'e inequality $\|u - u_E\|_{L^2} \lesssim \| \bar \nabla u \|_{L^2} $ (\cite{LieLos01}, Theorem 8.11). Also note that the fourth inequality in \eqref{eq:chain} requires the assumption $f\in W^{1,\infty}(E)$, while the last one uses $\essinf \lambda > 0$.

Let $\{u_n\} \subset H^1(E,\mathbb{R}^2)$ converge to $\hat u$ in the weak topology of $H^1(E,\mathbb{R}^2)$. In particular, $\bar{\nabla} u_n$ converges weakly in $L^2(E,\mathbb{R}^6)$ to the corresponding gradient $\bar{\nabla} \hat u$. By the compact embedding of $H^1(E,\mathbb{R}^2)$ into $L^2(E,\mathbb{R}^2)$, up to choosing a subsequence we may assume that $u_n$ also converges strongly in $L^2$ to $\hat u$. As is made apparent in (\ref{eq:sqrnormconvacc}), the expression $|D_u u|^2$ is convex in $\bar{\nabla} u$ for fixed $u$, and $\mathcal{E}$ is always nonnegative. We can therefore apply a standard result (\cite{Dac08}, Theorem 3.23) to obtain that $\mathcal{E}$ is lower semicontinuous.
\end{proof}
There is no reason to expect minimizers of $\mathcal{E}$ to be unique. The following example demonstrates that $\mathcal{E}$ is not convex in general.
\begin{exmp}
Let $v_1,v_2$ be the two vector fields from Example \ref{ex:zeroconacc} with $g$ being the identity, that is
\begin{equation*}
	v_1(t,x^1,x^2) = \begin{pmatrix}0 \\ x^1 \end{pmatrix}, \quad
	v_2(t,x^1,x^2) = \begin{pmatrix} x^2 \\ 0 \end{pmatrix}
\end{equation*}
They satisfy $D_{v_1} v_1 = D_{v_2} v_2 \equiv 0$, while
$$D_w w = \frac{1}{4} \begin{pmatrix} x^1 \\ x^2 \end{pmatrix},$$
where $w = (v_1+v_2)/2$. We conclude
\begin{equation*}
	0 = \|D_{v_1} v_1\|^2_{L^2} = \|D_{v_2} v_2\|^2_{L^2} < \|D_w w\|^2_{L^2} =\frac{T}{16} \int_\Omega |x|^2 \, dx
\end{equation*}
so that the functional $u \mapsto \|D_u u\|_{L^2}^2$ cannot be convex. This clearly implies that there are $\alpha, \beta, f$ such that $\mathcal{E}$ is not convex.
\end{exmp}

\subsection{Data term and contrast invariance.}\label{sec:dataterm}
Contrast invariance is a useful property of image processing operators. Let $h$ be a change of contrast, that is, a differentiable real function with $h' > 0$. An operator $A$ is called contrast invariant, if it commutes with all such contrast changes:
$$A \circ h (f) = h \circ A (f)$$
for all images $f$ (\cite{AlvGuiLioMor93}, Sec.~2.3).

A similar property can be postulated for operators estimating image motion, since an order-preserving rearrangement of the grey values of an image sequence should certainly not change velocities. Therefore, we call an operator $A$ that maps image sequences $f$ to velocity fields \emph{contrast invariant}, if it satisfies
\begin{equation*}
	A(f) = A \circ h (f).
\end{equation*}
However, a typical optical flow model of the form
\begin{equation}\label{eq:badmodel}
	f \mapsto \argmin_u \big\{ \|D_u f\|^p_{L^p} + \alpha \mathcal{R}(u) \big\}
\end{equation}
does not have this property. A simple counterexample is multiplication by a positive number $h(f) = cf$. For a model like \eqref{eq:badmodel} this change of contrast effectively amounts to dividing the regularization parameter $\alpha$ by $c^p$ and therefore definitely influences the result. There is, however, also a local effect of this contrast dependence, which is easy to confirm experimentally. Consider a scene with a dark background and two moving objects which are similar in shape and size and which move at similar velocities, but which have significantly different average brightnesses. This is precisely the case in the example of Figure \ref{fig:taxi}. Then, the velocities of the darker object, i.e.~the one with lower values of $f$, will be regularized more strongly than those of the brighter one, if a data term like the above is used.

\begin{figure}[ht!]
     \begin{center}
        \subfigure{
            \label{fig:taxi-input}
            \includegraphics[width=0.222\textwidth]{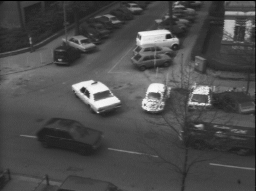}
        }
        \subfigure{
           \label{fig:taxi-unweighted-lr}
           \includegraphics[width=0.222\textwidth]{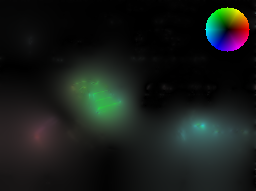}
        }
        \subfigure{
            \label{fig:taxi-unweighted}
            \includegraphics[width=0.222\textwidth]{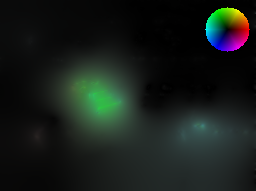}
        }
        \subfigure{
            \label{fig:taxi-weighted}
            \includegraphics[width=0.222\textwidth]{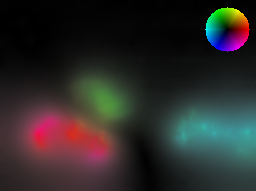}
        }
    \end{center}
    \caption{Effect of the data term weighting for $\mathcal{E}$ with $\alpha = 0$. When it is omitted, the data term inappropriately favors high-contrast objects. Shown are, from left to right, one frame of $f$, the corresponding results without weighting and $\beta=0.2$, without weighting and $\beta=0.5$, and weighted with $\omega = \sqrt{|\bar \nabla f |^2 + \epsilon^2}$ and $\beta=0.005$.}
   \label{fig:taxi}
\end{figure}

However, contrast dependence is not an inherent property of optical flow. Observe that, if $u$ solves $D_u f = 0$, then it also solves $D_u (h \circ f) = 0$, because $D_u (h \circ f) = h' D_u f$ and $h'>0$ by assumption. Therefore, it should rather be regarded a side-effect of the variational regularization approach. Contrast invariance can be restored by weighting the data term with $\lambda = 1/\omega$, where $\omega$ is a positively 1-homogeneous function that depends only on the first order derivatives of $f$. It then follows that
\begin{equation}\label{eq:continv}
	\frac{D_u (h \circ f)}{\omega(\bar \nabla (h \circ f))} = \frac{h' D_u f}{h' \omega(\bar \nabla f)} = \frac{D_u f}{\omega(\bar \nabla f)}.
\end{equation}

Such weighted data terms have already been used, although it seems with a different reasoning. In \cite{LaiVem98,ZimBruWei11}, for example, the weight was chosen to be $\omega = \sqrt{|\nabla f |^2 + \epsilon^2}$. In this paper we experiment with the weight $\omega = \sqrt{|\bar \nabla f |^2 + \epsilon^2}$. With this choice the data term essentially measures the orthogonal projection of $\bar u$ onto the span of $\bar \nabla f$, since
\begin{equation*}
	\frac{D_u f}{\omega} \approx \frac{\bar \nabla f}{|\bar \nabla f|} \cdot \bar u.
\end{equation*}
The effectiveness of this choice is illustrated in Figure \ref{fig:taxi}, in which the relative velocities of the vehicles are recovered correctly even though they have widely different contrast levels.

Finally, observe that weights which do not depend on first derivatives of $f$ (only), like $\omega = |f| + \epsilon$ for instance, do not yield full contrast invariance. While they work fine for the special case of rescaling $h(f) = cf$, they fail in general since the factor $h'$ in \eqref{eq:continv} does not cancel.

\section{Numerical solution}\label{sec:numerics}
We now seek to formulate a numerical approach for the minimization of the functional $\mathcal{E}$. As shown above, the convective regularization term is nonconvex and leads to nonlinear optimality conditions. Therefore, we propose an iterative scheme to arrive at an adequate flow field.
\paragraph{Iterative scheme.}
%

Consider the functional
\begin{equation*}\label{Ffunc}
	\mathcal{G}(u, w) = \|\lambda D_u f\|^2_{L^2} + \alpha \|D_w u\|^2_{L^2} + \beta \|\bar{\nabla} u\|^2_{L^2},
\end{equation*}
which satisfies $\mathcal{E}(u)=\mathcal{G}(u, u)$. For fixed $w \in L^{\infty}(E,\mathbb{R}^2)$, the mapping $u \mapsto \mathcal{G}(u, w)$ is a convex and quadratic functional, and $\mathcal{G}(u, w) < +\infty$ for any $u \in H^1(E)$. Our iterative method then reads as follows.
\begin{itemize}
\item Given an image sequence $f$ and parameters $\beta_0,\alpha_1, \beta_1$:
\item Find an initial guess $u_0$ by minimizing $\mathcal{G}(u, 0)$ over $u$, with $\alpha = 0, \beta = \beta_0$. This induces isotropic regularization in time and space, corresponding to a time-regularized variant of the Horn-Schunck model.
\item Compute $u_{k+1}$ by minimizing over the variable $u$ the functional $\mathcal{G}(u, u_k)$ with parameters $\alpha = \alpha_1, \beta = \beta_1$.
\end{itemize}
Each of these steps corresponds to finding an optical flow field using anisotropic regularization with the diffusion tensor $\alpha\bar{w}\bar{w}^\top + \beta \textrm{Id}$, where $\textrm{Id}$ is the $3 \times 3$ identity matrix. Therefore, this particular scheme is based on the diffusion interpretation of the convective regularization term, as reflected in \eqref{eq:sqrnormconvacc}. Indeed, upon fixing $w$ the first term in the complete Euler-Lagrange equation \eqref{eq:convectiveeullag} does not appear. Even though convergence of the iterative scheme is not guaranteed, in all of our experiments the solution $u_k$ stabilized after only a few iterations. Typically by $k=3$ or $4$ the difference between successive iterates became negligible.

By the properties of $\mathcal{G}$ listed above, each of these minimization problems can be performed by solving the corresponding linear Euler-Lagrange equation
\begin{equation}\label{eq:numericeullag}
	\bar \nabla \cdot \left( \bar{\nabla}u ( \alpha \bar{w} \bar{w}^\top + \beta \mathrm{Id}) \right) - \lambda^2 (\nabla f \cdot u)\nabla f^\top = \lambda^2 f_t \nabla f^\top,
\end{equation}
coupled with natural boundary conditions
\begin{equation*}
	n \cdot \left( \bar{\nabla}u ( \alpha \bar{w} \bar{w}^\top + \beta \mathrm{Id})\right) = 0 \quad \text{on } \partial E,
\end{equation*}
where $n$ is the outward normal to $E$. Our approach of discretizing \eqref{eq:numericeullag} is specified next. 

\paragraph{Discretization.}
The model has been discretized using a multilinear finite element formulation on a regular rectangular grid, so that each vertex in the grid corresponds to one pixel of one video frame. The grid used is such that the spacing between nodes in space is one, while the spacing in time (from frame to frame) may be smaller. We emphasize that the relation of the two grid step sizes implicitly controls the relative amount of regularization done in time and space, since the derivatives of the functional scale accordingly. Usually, using the same spacing will lead to results that are highly over-regularized in time. In the presented examples, the time step was divided by a factor of $8$ with respect to that of space.

Both the input $f$ and the unknown vector field $u$ are represented by their nodal values, leading to continuous piecewise multilinear functions. Since the Euler-Lagrange equation \eqref{eq:numericeullag} is actually a system for the two components $u^1, u^2$, the resulting linear system is represented by a block matrix, with the different blocks describing the interactions between components. The entries of these blocks are assembled from the integrals over different elements arising in the weak formulation, as per finite element practice \cite{QuaVal08}. Such integrals are computed exactly through an adequate Gauss quadrature. Note that boundary conditions do not need to be imposed explicitly.

A consequence of representing $f$ by a piecewise multilinear function is that the corresponding partial derivatives $f_{x^i}, f_t$ appearing in the right hand side $\lambda^2 f_t \nabla f^\top$ of \eqref{eq:numericeullag} do not belong to the same finite element space as $f$. However, our finite element matrices act on the nodal values of functions in this space, so an interpolation step is needed to obtain functions with the appropriate smoothness. To this end, we have performed an $L^2$ projection of these partial derivatives into the space of piecewise multilinear functions, whose nodal values are then multiplied together and used as the right hand side for our linear system. Computationally each projection requires the solution of an additional linear system, which needs to be done only once.

\section{Experiments}\label{sec:experiments}
In Figures \ref{fig:traffic}, \ref{fig:passat} and \ref{fig:passat-crop} we present some examples of flows computed through the numerical scheme of the previous section, approximating minimizers of 
\begin{equation*}
\mathcal{E}(u) = \|\lambda D_u f\|^2_{L^2} + \alpha \|D_u u\|^2_{L^2} + \beta \|\bar \nabla u\|^2_{L^2},
\end{equation*}
which are compared with corresponding results with basic isotropic regularization, that is, minimizers of 
\begin{equation*}
\mathcal{H}(u) = \|\lambda D_u f\|^2_{L^2} + \beta \|\bar \nabla u\|^2_{L^2}.
\end{equation*}
In all these examples results were computed over an interval of $30$ frames.

In the sequences used, the time derivative of the vector field at a given position is likely not appropriate, since isolated objects travel across the images. In the results with isotropic regularization, some of the disadvantages of basic quadratic regularizers are apparent. With lower regularization parameters (Figures \ref{fig:traffic-lowreg} and \ref{fig:passat-lowreg}) inner edges of rigidly moving objects are visible. With higher regularization parameter values, the support of the vector field is enlarged and interactions between distinct objects moving ensue, creating cancellations when the objects move in different directions (Figure \ref{fig:traffic-hireg}) or artificial reinforcement when they move in similar directions (Figure \ref{fig:passat-hireg}). However, with convective regularization both of these disadvantages can be avoided (Figures \ref{fig:traffic-convective} and \ref{fig:passat-convective}).

Additionally, one needs to choose the regularization constant $\epsilon$ for the weight $\sqrt{| \bar \nabla f|^2 +\epsilon^2}$ in the data term. Clearly, choosing a too large $\epsilon$ will diminish the effect of the weighting scheme, while choosing it too small may overemphasize very low contrast regions and potentially amplify noise. Since the 8-bit input images were normalized so that their values are of the form $k/255$ with $k=0\ldots 255$, we chose $\epsilon=0.01$. That is, $\epsilon$ is roughly of the size of the smallest nonzero derivative that may appear in the data.
\begin{figure}[ht!]
     \begin{center}
        \subfigure[Input]{
            \label{fig:traffic-input}
            \includegraphics[width=0.45\textwidth]{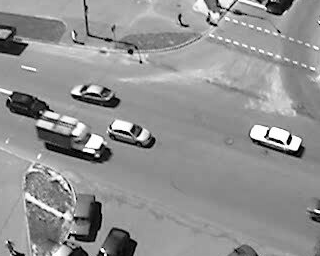}
        }
        \subfigure[Result with $\mathcal{H}$, $\beta=5e-4$]{
           \label{fig:traffic-lowreg}
           \includegraphics[width=0.45\textwidth]{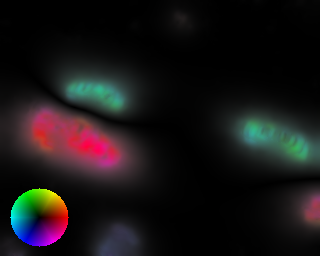}
        }\\
        \subfigure[Result with $\mathcal{H}$, $\beta=5e-3$]{
            \label{fig:traffic-hireg}
            \includegraphics[width=0.45\textwidth]{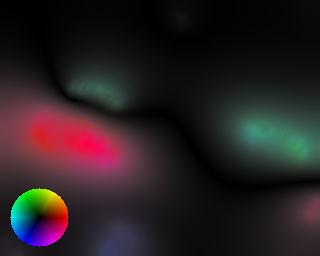}
        }
        \subfigure[Result with $\mathcal{E}$, $\alpha=5e-3, \beta=5e-4$]{
            \label{fig:traffic-convective}
            \includegraphics[width=0.45\textwidth]{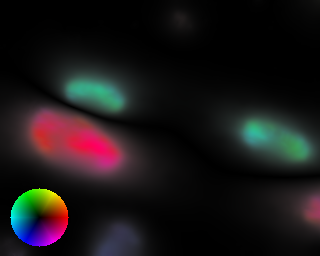}
        }
    \end{center}
    \caption{Comparison of isotropic and convective regularization.}
   \label{fig:traffic}
\end{figure}

\begin{figure}[ht!]
     \begin{center}
        \subfigure[Input]{
            \label{fig:passat-input}
            \includegraphics[width=0.45\textwidth]{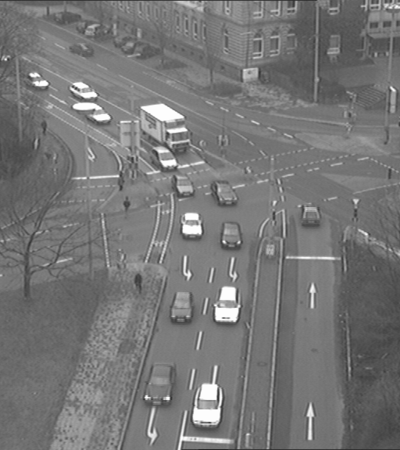}
        }
        \subfigure[Result with $\mathcal{H}$, $\beta=5e-5$]{
           \label{fig:passat-lowreg}
           \includegraphics[width=0.45\textwidth]{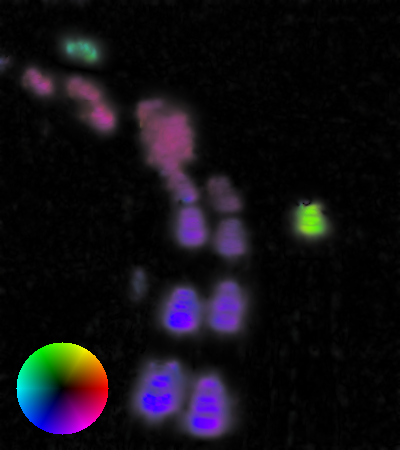}
        }\\
        \subfigure[Result with $\mathcal{H}$, $\beta=5e-4$]{
            \label{fig:passat-hireg}
            \includegraphics[width=0.45\textwidth]{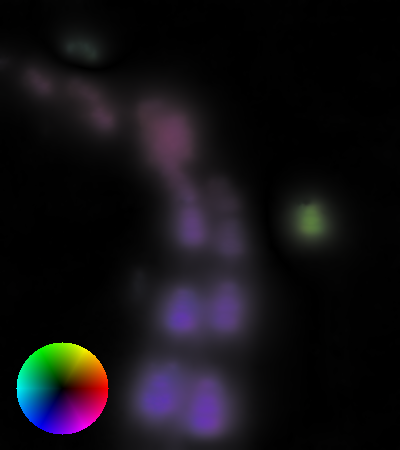}
        }
        \subfigure[Result with $\mathcal{E}$, $\alpha=1e-3, \beta=5e-5$]{
            \label{fig:passat-convective}
            \includegraphics[width=0.45\textwidth]{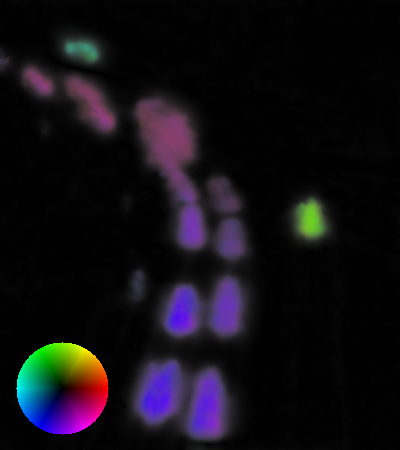}
        }
    \end{center}
    \caption{Comparison of isotropic and convective regularization.}
   \label{fig:passat}
\end{figure}

\begin{figure}[ht!]
     \begin{center}
        \subfigure{
            \label{fig:passat-input-crop}
            \includegraphics[width=0.22\textwidth]{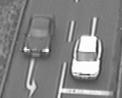}
        }
        \subfigure{
           \label{fig:passat-lowreg-crop}
           \includegraphics[width=0.22\textwidth]{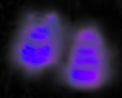}
        }
        \subfigure{
            \label{fig:passat-hireg-crop}
            \includegraphics[width=0.22\textwidth]{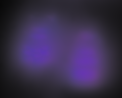}
        }
        \subfigure{
            \label{fig:passat-convective-crop}
            \includegraphics[width=0.22\textwidth]{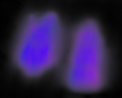}
        }
    \end{center}
    \caption{Cutout from Figure \ref{fig:passat}. The use of convective regularization allows to obtain nearly-constant optical flow field inside the objects without excessively enlarging the support.}
   \label{fig:passat-crop}
\end{figure}
In our experiments the parameter $\beta_1$ may be chosen quite small to preserve the boundaries of objects, while using a relatively large $\alpha_1$ for the convective regularization term. By using the convective derivative, the information from other frames is used in a consistent way to correctly fill the flow field inside objects even in the absence of texture. The initial isotropic parameter $\beta_0$ for starting the iterative scheme was chosen to be $\beta_0=\alpha_1$ in all cases.

\section{Conclusion}
In this article we presented an entirely continuous optical flow model that is based on the assumption that velocities should vary smoothly not at fixed points in space but along motion trajectories. The resulting regularization term has a variety of different interpretations. First, this term penalizes the proper acceleration of objects moving according to the velocity field. Second, it has a direct relation to the curvature of motion trajectories. Finally, when minimized it acts on the vector field not only as a projection term but also as an edge-preserving anisotropic regularizer.

We tested the convective regularization term by incorporating it into a reference optical flow model and comparing the results with and without the new term. In order to make the changes as apparent as possible and to avoid potential side effects we tried to keep the complexity of the reference model at a minimum. Therefore, the resulting functional \eqref{eq:functional} should mainly be viewed as a prototype of an optical flow model with convective regularization which can be improved upon by future research.

\paragraph{Acknowledgements.}
This work has been supported by the Austrian Science Fund (FWF) within the national research network Geometry + Simulation (project S11704, Variational Methods for Imaging on Manifolds). Additionally, JI acknowledges support from the FWF Doctoral Program Dissipation and Dispersion in Nonlinear PDEs (W1245). We wish to thank Enrique C.~Vargas for inspiring exchanges. Our finite element implementation was carried out in the \nolinkurl{Quocmesh} library,\footnote{\nolinkurl{http://numod.ins.uni-bonn.de/software/quocmesh/}} AG Rumpf, Institute for Numerical Simulation, Universit\"{a}t Bonn. For our experiments, we have used footage from the datasets of Lund University\footnote{\nolinkurl{http://www.tft.lth.se/video/co_operation/data_exchange/}} (sequence ``Conflicts-20100628-1344, camera 1") and Universit\"{a}t Karlsruhe\footnote{\nolinkurl{ftp://ftp.ira.uka.de/pub/vid-text/image_sequences/}} (sequences ``taxi" and ``dt\_passat\_part2").

\bibliographystyle{plain}
\bibliography{\BibPath strings,\BibPath articles,\BibPath books,\BibPath infmath,\BibPath infmath_books,\BibPath infmath_report,\BibPath infmath_talks,\BibPath infmath_theses,\BibPath inproceedings,\BibPath preprints,\BibPath proceedings,\BibPath theses,\BibPath unsubmitted,\BibPath websites}

\end{document}